\documentclass[12pt]{amsart}
\usepackage{amsfonts}

\newtheorem{theorem}{Theorem}[section]
\newtheorem{corollary}[theorem]{Corollary}

\newtheorem{lemma}[theorem]{Lemma}
\newtheorem{proposition}[theorem]{Proposition}
\newtheorem{example}[theorem]{Example}

\def\NN{\hbox{\sf I\kern-.13em\hbox{N}}}
\def\RR{\hbox{\sf I\kern-.14em\hbox{R}}}

\def\Cc{\hbox{\sf C\kern -.47em {\raise .48ex \hbox{$\scriptscriptstyle |$}}
   \kern-.5em {\raise .48ex \hbox{$\scriptscriptstyle |$}} }}

\newcommand{\cS}{{\mathcal S}}
\newcommand{\cC}{{\mathcal C}}
\newcommand{\cD}{{\mathcal D}}

\newcommand{\cM}{{\mathcal M}}

\begin{document}

\baselineskip 7mm

\title{From local to global ideal-triangularizability\footnote{The paper will appear in Linear and Multilinear Algebra.}}
\author{Roman Drnov\v{s}ek, Marko Kandi\'{c}}
\date{\today}

\begin{abstract}
\baselineskip 7mm
Let $L$ be a Banach lattice with order continuous norm, and let $\cS$ be a multiplicative semigroup of ideal-triangularizable positive 
compact operators on $L$ such that, for every pair $\{S, T\} \subseteq \cS$, the atomic diagonal of the commutator $S T - T S$ is equal to zero.
We prove that the semigroup $\cS$ is ideal-triangularizable. 
\end{abstract}

\maketitle

\noindent
{\it Math. Subj. Classification (2000)}: 47A15, 47B65. 
 
\baselineskip 7mm

\section{Introduction}
\vspace{5mm}

Let $L$ be a normed Riesz space, and let $L^+$ denote the positive cone of $L$. Elements of $L^+$ are called {\it positive vectors}. A linear subspace $J$ of $L$ is an ideal in $L$ whenever 
$|x|\leq |y|$ and $y\in J$ imply $x\in J$. A {\it band} in $L$ is an order closed ideal. A { band $B$ in $L$ is a {\it projection band} whenever $L=B\oplus B^d$. In this case, there exists a positive projection onto $B$ with a kernel $B^d$, and the direct sum is also an order direct sum.  A  Banach lattice $L$ is said to have an {\it order continuous norm} whenever every decreasing net $\{x_\alpha\}_\alpha$ with infimum $0$ of positive vectors in $L$ converges in norm to $0$. It is well known that a Banach lattice with order continuous norm is Dedekind complete, and that its 
closed ideals are bands \cite{AlAp}. Moreover, every band in a Dedekind complete Banach lattice is a projection band. 

A non-zero vector $a\in L^+$ is an {\it atom} in a normed Riesz space $L$ if $0 \le x, y \le a$ and $x \land y = 0$ imply 
either $x=0$ and $y=0$, or equivalently, if $0 \leq x \leq a$ implies $x=\lambda a$ for some $\lambda \geq 0$, i.e., 
the principal ideal $B_a$ generated by $a$ is one dimensional. It turns out that $B_a$ is a projection band \cite{LuxZaa}. The decomposition $L=B_a\oplus B_a^d$ implies that for an arbitrary (positive) vector $x\in L$ there exist a  (positive) scalar $\lambda_x$ and a (positive) vector $y_x\in B_a^d$ such that $x=\lambda_x a+y_x.$ The linear functional $\varphi_a: L\to \mathbb R$ associated to the atom $a$ is 
defined by $\varphi_a(x)=\lambda_x$.
A positive vector $x\in L$ is a {\it quasi-interior point} in $L$ if the principal ideal in $L$ generated by $x$ is dense in $L$. 

A (linear) operator $T$ between Riesz spaces $L_1$ and $L_2$  is said to be {\it positive} if $T$ maps the positive cone $L_1^+$ into the positive cone $L_2^+$. 
The {\it absolute kernel} $\mathcal N(T)$ of $T$ is defined as
$\mathcal N(T) = \{x\in L_1:\; T|x|=0\}$. If it is a zero ideal, then $T$ is said to be {\it strictly positive}.
The {\it range ideal} $\mathcal R(T)$ of $T$ is the ideal generated by the range of $T$.
Recall that every positive operator $T$ on a Banach lattice is continuous, and its spectral radius $r(T)$ belongs to its spectrum $\sigma(T)$.
Throughout the text we assume that functionals and operators acting on normed Riesz spaces are continuous. 

A family $\mathcal F$ of operators on a normed Riesz space $L$ is said to be {\it ideal-reducible} whenever there exists a nontrivial closed ideal in $L$ that is invariant under all operators from the family $\mathcal F$. 
Otherwise, we say that $\mathcal F$ is {\it ideal-irreducible}. 
The following very useful proposition is proved in \cite{DKan}.

\begin {proposition}\label{o razcepnosti}
Let $L$ be a normed Riesz space, and let $\mathcal S$ be a semigroup of positive operators on $L$. 
The following statements are equivalent:
\begin{enumerate}
	\item [(a)] $\mathcal S$ is ideal-reducible;
	\item [(b)] there exist a nonzero positive functional $\varphi \in L^*$ and a nonzero positive vector $f\in L^+$ 
	such that $\varphi(\cS f)=\{0\}$; 
	\item [(c)] there exist nonzero positive operators $A$ and $B$ on $L$ such that $A\mathcal SB=\{0\}$;
	\item [(d)] some nonzero semigroup ideal of $\mathcal S$ is ideal-reducible.
\end{enumerate}
 \end {proposition}

If there is a chain $\mathcal C$ that is maximal as a chain of closed ideals of $L$ and that has the property that  every ideal in $\mathcal C$ is invariant under all the operators in a family $\mathcal F$, then $\mathcal F$ is said to be {\it ideal-triangularizable}, and $\mathcal C$ is an {\it ideal-triangularizing chain} for $\mathcal F$. Every maximal chain of closed ideals is also maximal as a chain of closed subspaces of $L$ \cite{Drn}. Let $I$ and $J$ be closed ideals in $L$ that are invariant under every operator from a family $\mathcal F$.
If $I \subseteq J$, then $\mathcal F$ induces a family $\widehat{\mathcal F}$ of operators on the quotient normed Riesz space $J/I$ as follows. 
For each $T\in \mathcal F$ the operator $\widehat T$ is defined on $J/I$ by 
$\widehat T(x+I)=Tx+I.$ Any such family $\widehat{\mathcal F}$ is called a {\it family of ideal-quotients} of the family $\mathcal F$. 
A set $\mathcal P$ of properties is said to be {\it inherited by ideal-quotients} if every family of ideal-quotients of a family of operators satisfying $\mathcal P$ also satisfies the same properties. 
Ideal-triangularizability of operators on Banach lattices is in practice often reduced to ideal-reducibility of operators. 
The details are contained in the following lemma that was proved in \cite{Drn1}.

\begin {lemma}  [The Ideal-triangularization Lemma]\label{trikotljivostna lema}
Let $\mathcal P$ be the set of properties inherited by ideal-quotients. If every family of operators on a Banach lattice of dimension greater than one which satisfies $\mathcal P$
is ideal-reducible, then every such family is ideal-triangularizable. 
\end {lemma}

Let $T$ be an operator on a Banach lattice $L$ with an order continuous norm. 
If $B_1$ and $B_2$ are bands of $L$ such that $B_2 \subseteq B_1$, then the quotient Banach lattice 
$B_1/B_2$ is isometrically lattice isomorphic to the band $B := B_1 \cap B_2^d$, and so the norm of $B$ is also order continuous.
If $T$ leaves the bands $B_1$ and $B_2$ invariant, then the induced operator $\hat{T}$ on $B_1/B_2$ can be identified 
with the compression of $T$ on $B$, that is, with the operator $P_B T|_B$,  where $P_B$ denotes the band projection on $B$. 
   
For the terminology not explained in the text about normed Riesz spaces or Banach lattices and operators acting on them we refer the reader to classical textbooks \cite{AbAl}, \cite{AlAp}, \cite{Sch} and \cite{LuxZaa}.

Finally, we explain the organization of the paper. Section 2 deals with the structure of positive idempotents of finite ranks.
In Section 3 the atomic diagonal operator $\mathcal D$ introduced recently in \cite{MKP} is connected with the known band projection 
onto the center of a Dedekind complete Banach lattice. 
The main result of the paper is Theorem \ref{semigroup1} 
(and its finite-dimensional special case Corollary \ref{matrices}) in Section 4.  It reveals the relation between 
the atomic diagonal operator and the ideal-triangularizability of semigroups of positive operators.
We also prove Theorem \ref{complete_decomposability2} that completes \cite[Theorem 3]{MKP} and extends \cite[Theorem 3.8]{BMR} and 
\cite[Corollary 27]{MR05}.

\section {The structure of positive idempotents of finite rank}

The structure result for positive idempotents on $L^p$-spaces ($1\leq p<\infty$) was given by Zhong in \cite{Zhong};
see also the book \cite[Section 8.7]{RaRo} for a complete treatment. 
In this section we extend this result to Banach lattices with order continuous norm.

Let $L$ be a normed Riesz space, and let $P$ be a positive idempotent operator of rank one on $L$. 
Then there exist a positive vector $x \in L$ and a positive linear functional $\varphi$ on $L$ such that $\varphi(x) = 1$ and $P=x\otimes \varphi$, i.e., 
$Py=\varphi(y)x$ for $y \in L$. The ideal-irreducibility of such an idempotent is characterized by the following lemma that is a generalization of \cite[Lemma 8.7.11]{RaRo}.

\begin {lemma}\label{lema o idempotentu}
Let $L$ be a normed Riesz space, and let $P=x\otimes \varphi$ be a positive idempotent of rank one on $L$, where 
$x \in L$ is a positive vector and $\varphi$ a positive linear functional on $L$. The following statements are equivalent:
\begin {enumerate}
\item [(a)] $P$ is ideal-irreducible;
\item [(b)] The vector $x$ is a quasi-interior point of $L$, and $\varphi$ is a strictly positive functional on $L$;
\item [(c)] $P$ and $P^*$ are strictly positive operators on $L$ and $L^*$, respectively. 
 \end {enumerate}
\end {lemma}

\begin {proof}
	To see that (a) implies (b), assume that $P$ is ideal-irreducible. 
	If $y$ is a positive vector in the absolute kernel $\mathcal N(\varphi)$ of $\varphi$, then 
	$Py=\varphi(y)x=0$  implies that $\mathcal N(\varphi)$ is invariant under $P$. Since $P$ is ideal-irreducible, we have $\mathcal N(\varphi)=\{0\},$ so that $\varphi$ is strictly positive on $L$. 
	If $0\leq |y|\leq \lambda x$ with $y\in L$ and $\lambda\geq 0$, then
	$|Py|\leq P|y|\leq \lambda Px=\lambda \varphi(x)x$ implies that the closure of the principal ideal  generated by the positive vector $x$ is invariant under $P$.
	Hence, $x$ needs to be a quasi-interior point in $L$. 

      Assume that (b) holds. To prove (a), assume that $J$ is a proper closed ideal that is invariant under the operator $P$. 
	By \cite[Lemma 1.1]{DKan}, there exists a nonzero positive functional $\psi$ on $L$ which is zero on $J$. 
      Then for  $y \in J$ we have $0=\psi(Py)=\psi(x) \varphi(y)$. Since $x$ is a quasi-interior point,  $\psi(x) > 0$ by \cite[Lemma 4.15]{AbAl}, and so we conclude that  $\varphi(y)=0$.
      Since $\varphi$ is a strictly positive functional, this implies that $y=0$, so that  $J = \{0\}$. Therefore, the operator $P$ does not have nontrivial closed invariant ideals.  

      Assume again that (b) holds. To show (c), assume that $Py=0$ and $P^* \psi=0$ for some positive vectors $y$ and $\psi$ in $L$ and $L^*$, respectively.
	Since $\varphi$ is a strictly positive functional on $L$, the equality $0 = Py=\varphi(y)x$ implies that $y=0$. 
	Since $x$ is a quasi-interior point in $L$, the equality $0 = P^*(\psi)=\psi(x)\varphi$ implies that $\psi=0$ by \cite[Lemma 4.15]{AbAl}.

	To see that (c) implies (b), suppose that $P$ and $P^*$ are strictly positive operators on $L$ and $L^*$, respectively.
	Let $\psi$ be an arbitrary positive functional on $L$ with $\psi(x)=0$. 	Then 
	$P^*\psi=(\varphi\otimes x)\psi=\psi(x)\varphi=0$ implies $\psi=0$, since $P^*$ is strictly positive. It follows from \cite[Lemma 4.15]{AbAl} that $x$ is a quasi-interior point in $L$.  
	Suppose now that $\varphi(y)=0$ for some positive vector $y\in L$.  Then $Py=(x\otimes \varphi)(y)=\varphi(y)x=0$. Since $P$ is strictly positive, we have $y=0$, so that $\varphi$ is strictly positive on $L$. 
\end {proof} 

Note that in the case when any of the equivalent statements of Lemma \ref{lema o idempotentu} holds for a positive ideal-triangularizable idempotent $P$ of rank one, 
then $L$ is lattice isomorphic to $\mathbb R$, and $P$ is just the identity on $\mathbb R$. 

Positive idempotents of finite rank are treated in the following propositions that extend  \cite[Proposition 8.7.12]{RaRo}. 
It should be noted that these propositions contain additional statements regarding the structure of finite rank idempotents that are also ideal-triangularizable. 

\begin {proposition}\label{o idempotentu 2}
	Let $L$ be a Banach lattice with order continuous norm and $P$ a positive idempotent of finite rank $r$ such that 
	$P$ and $P^*$ are strictly positive. There exist pairwise disjoint bands $L_1,\ldots, L_r$ in $L$ such that 
	$L=L_1\oplus \cdots \oplus L_r$ and $P$ is of the form 
	$$P=P_1\oplus \cdots \oplus P_r,$$ where $P_j$ is an ideal-irreducible positive idempotent of rank $1$ on $L_j.$
	If, in addition, $P$ is ideal-triangularizable, then $r=\dim L$ and $P$ is the identity operator on $L$. 
\end {proposition}

\begin {proof}
	If $r=1$, then the assertion follows from Lemma \ref{lema o idempotentu} and the remark following it. Assume $r>1.$
	Let $x$ and $y$ be linearly independent positive vectors in the range of the operator $P.$
	Replacing $y$ by $x+y$ (if necessary) we may assume $0\leq x\leq y$ and $x \neq y$.
	Let $t_0$ be the supremum of the nonempty set $\{t\in \mathbb R:\; t\geq 0, \,ty\le x\}$ that is bounded from above.  
      Since $t_0 y \le x \le y$ and $x \neq y$, we have $t_0<1$. Pick $s \in (t_0, 1)$ and let $z=x-sy.$
	The vectors $z^+$ and $z^-$ are both nonzero, since $z$ is neither positive nor negative. 
	From 
	$z=Pz=Pz^+-Pz^-\leq Pz^+$ it easily follows $z^+\leq Pz^+$. Since $P$ is strictly positive, $P(Pz^+-z^+)=0$ implies $Pz^+=z^+.$
	Let $J$ be the closed ideal generated by the vector $z^+.$  Note that $J$ is invariant under $P$ and it is nontrivial, since $z^-$ is nonzero. 
	Therefore, the operator $P$ can be decomposed as
	$$\left[\begin {array}{cc}
	P_1 & P_2\\
	0 & P_4\\
	\end {array}
	\right]$$ with respect to the decomposition $L=J \oplus J^d$.

	We claim that $P_2=0$. Since  $P^2=P$, we have $P_1^2=P_1$, $P_4^2=P_4$ and $P_1P_2+P_2P_4=P_2.$ We conclude that $P_1P_2P_4=0.$ 
	Since the operator $P$ is strictly positive, the operator $P_1$ is also strictly positive, and so $P_2P_4=0.$ Taking the adjoints we have $P_4^*P_2^*=0$. 
	Since the adjoint $P^*$ is strictly positive, so is the operator $P_4^*$ on $(J^d)^*$, and hence $P_2=0$ as claimed.  	
	We finish the proof  by induction on $r$. 
\end {proof}

\begin {proposition}\label{o idempotentu 3}
		Let $L$ be a Banach lattice with order continuous norm and $P$ a positive idempotent of finite rank $r$ on $L$. 
		There exist pairwise disjoint bands $B_1$, $B_2$ and $B_3$ with $L=B_1\oplus B_2\oplus B_3$
		such that $P$ can be decomposed as 
  \begin{equation}
  \label{decomp}
		P = \left[\begin {array}{ccc}
		0 & XQ & XQY\\
		0 & Q & QY\\
		0 & 0 & 0\\
		\end {array}\right] = 
		\left[\begin {array}{c}
		X \\
		I \\
		0 \\
		\end {array}\right] Q 
		\left[\begin {array}{ccc}
		0 & I & Y \\
		\end {array}\right],
  \end{equation}		
		where $Q$ is an idempotent of rank $r$ such that 
		$Q$ and $Q^*$ are strictly positive operators on $B_2$ and $B_2^*$, respectively. 
		Moreover, there exist pairwise disjoint bands $L_1,\ldots, L_r$ in $B_2$ with 
		$B_2=L_1\oplus \cdots \oplus L_r$ such that $Q=Q_1\oplus \cdots \oplus Q_r$, where 
		$Q_j$ is a strictly positive idempotent of rank one on $L_j.$
		If, in addition, $P$ is ideal-triangularizable, then the dimension of $B_2$ is finite and $Q$ is the identity operator on $B_2.$
	\end {proposition}
	
	\begin {proof}
		Let $B_1$ be the absolute kernel of the operator $P$, and let 
                     $B$ be the band generated by $B_1$ and the range of the operator $P$.
		Let $B_2=B\cap B_1^d$ and $B_3=B^d.$
		Let $$\left[\begin {array}{ccc}
		0 & X & Z\\
		0 & Q & Y\\
		0 & 0 & 0\\
		\end {array}\right]$$
		be the operator matrix corresponding to the operator $P$ with respect to the decomposition $L=B_1\oplus B_2\oplus B_2$.
		
		Since $P$ is an idempotent, we have $Q^2=Q$, $X=XQ$, $Y=QY$ and $Z=XY$, so that the decomposition (\ref{decomp}) is proved.
		It also shows that the rank of $Q$ must be $r$. 
		We claim that $Q$ and $Q^*$ are strictly positive operators on $B_2$ and $B_2^*$, respectively. 
		Suppose that for some positive vector $x\in B_2$ we have $Qx=0.$ Then 
		$Xx=XQx=0$, so that $Px=0$ implies $x=0$. Suppose that  a positive functional $\varphi$  in $B_2^*$ satisfies $Q^*\varphi=0.$
		Then $\varphi$ is zero on the range of the operator $P$ which implies $\varphi=0$.
		The last two assertions hold by Proposition \ref{o idempotentu 2}.
\end {proof}

\section {The atomic diagonal operator}

Throughout this section, let $L$ be a Dedekind complete Banach lattice. 
Denote by $\mathcal L_r(L)$ the Dedekind complete Riesz space of all regular operators on $L$, i.e., those operators that are linear combinations of positive operators. It is well known that $\mathcal L_r(L)$ becomes a Banach lattice algebra with respect to the regular norm defined by 
$\|T\|_r:=\||T|\|.$ The center $\mathcal Z(L)$  is the ideal in $\mathcal L_r(L)$ generated by the identity operator $I$, i.e., 
$$\mathcal Z(L)=\{T\in \mathcal L_r(L):\; |T|\leq  \lambda I \textrm{ for some }\lambda\geq 0\}.$$
If $T\in \mathcal Z(L)$, then the operator norm and the regular norm of $T$ coincide. Since $\mathcal Z(L)$ is also a band in $\mathcal L_r(L)$, we have a band decomposition 
$\mathcal L_r(L)=\mathcal Z(L)\oplus \mathcal Z(L)^d.$
Let $\mathcal P$ be the band projection onto $\mathcal Z(L).$ 
By a result of Voigt \cite{Vo88}, $\mathcal P$ is a contraction with respect to the operator norm. 
Schep \cite{Sc80} proved that the component $\mathcal P(T)$ of a positive operator $T$ in $\mathcal Z(L)$ is 
\begin {equation}\label{schepF}
\mathcal P(T)=\inf\left\{\sum_{i=1}^n P_iTP_i:\; 0\leq P_i\leq I,\, P_i^2=P_i,\, \sum_{i=1}^n P_i=I\right\}.
\end {equation}
 
Let $A$ be the band generated by all atoms in $L$, and let $\mathcal A \subseteq A$ be the maximal set of pairwise disjoint atoms of norm one. 
Given $a \in \mathcal A$, we denote by $P_a$ the band projection onto the band $B_a$.
Let $T$ be a positive operator on $L$. In  \cite{MKP} it is proved that the operator 
$$\mathcal D(T)=\sup\left\{\sum_{a\in \mathcal F}P_aTP_a:\; \mathcal F \textrm{ is a finite subset of }\mathcal A\right\}$$
exists, since it is a supremum of an increasing net that is bounded from above,  and it also satisfies $0\leq \mathcal D(T)\leq T.$ 
If $L$ is atomic (i.e., $A = L$), then $\mathcal D(T) = \mathcal P(T)$. For a general $L$ we have

\begin {proposition}\label{atom diag predpis}
Let $T$ be a positive operator on $L$. If $P_A$ denotes the band projection onto the band $A$, 
then $\mathcal D(T)=P_A\mathcal P(T)$. 
\end {proposition}

\begin {proof}
Since $P_A \in \mathcal Z(L)$, it commutes with $\mathcal P(T)$, and so  $P_A\mathcal P(T)=P_A\mathcal P(T)P_A.$
Applying Schep's formula (\ref{schepF}) we obtain 
$$ P_A\mathcal P(T)=P_A\inf\left\{\sum_{i=1}^n P_iTP_i:\; 0\leq P_i\leq I,\, P_i^2=P_i,\, \sum_{i=1}^n P_i=I\right\}P_A .$$ 
Since $0\leq P_A\leq I$, $P_A$ is order continuous, so that 
$$ P_A\mathcal P(T)= \inf\left\{\sum_{i=1}^n (P_A P_i) T (P_i P_A):\; 0\leq P_i\leq I,\, P_i^2=P_i,\, \sum_{i=1}^n P_i=I\right\} = $$ 
$$ = \inf\left\{\sum_{i=1}^n Q_i T Q_i : \; 0\leq Q_i\leq P_A,\, Q_i^2=Q_i,\, \sum_{i=1}^n Q_i=P_A\right\} . $$ 
Since $A$ is an atomic Dedekind complete Banach lattice, $P_A\mathcal P(T) = \mathcal D(T)$.
\end {proof}

We use the equality from Proposition \ref{atom diag predpis} to extend the operator $\mathcal D$ to the operator on the whole space $\mathcal L_r(L)$, so we define ${\mathcal D}(T) := P_A\mathcal P(T)$ for $T \in \mathcal L_r(L)$. 
This extension is called the {\it atomic diagonal operator} and the operator $\mathcal D(T)$ on $L$ is said to be the {\it atomic diagonal} of an operator $T \in \mathcal L_r(L)$. 

\begin {proposition}\label{zvez diag}
The following assertions hold for the atomic diagonal operator $\mathcal D$:

(a) $\mathcal D$ is a band projection onto the band 
$$  \{T\in \mathcal L_r(L):\; |T|\leq \lambda P_A \textrm{ for some }\lambda \geq 0\}$$
in $\mathcal L_r(L)$ that can be identified with the center ${\mathcal Z}(A)$;

(b) $\|\cD(T)\| \leq \|\mathcal P(T)\| \leq \|T\|$ for all $T\in \mathcal L_r(L)$;

(c) $\mathcal P(T)=\mathcal D(T)$ for every positive compact operator $T$ on $L$.
\end {proposition}

\begin {proof}
For the proof of (a) just note that, with respect to the band decomposition $L=A \oplus A^d$, an operator $T\in \mathcal L_r(L)$, its modulus $|T|$ 
and the band projection $P_A$ have the block forms 
$$ 
T = \left[ \begin {array}{cc}
            T_1 & T_2 \\
            T_3 & T_4 \\
\end {array} \right]  
\ \ , \ \ \ \ 
|T| = \left[ \begin {array}{cc}
            |T_1| & |T_2| \\
            |T_3| & |T_4| \\
\end {array} \right] 
\ \ \ {\rm and} \ \ \ \ 
P_A = \left[ \begin {array}{cc}
            I  & 0 \\
            0 & 0 \\
\end {array} \right] .  
$$
To show (b), use Voigt's theorem \cite[Theorem 1.4]{Vo88} to obtain that 
$\|\mathcal D(T)\| = \|P_A\mathcal P(T)\|\leq \|\mathcal P(T)\| \le \|T\|$.

If $P_{A^d}$ is the band projection on $A^d$, then the operator 
$$ \mathcal P(T) - \mathcal D(T) = P_{A^d} \mathcal P(T) = P_{A^d} \mathcal P(T) P_{A^d} \in \mathcal Z(L) $$
is dominated by a positive compact operator $P_{A^d}TP_{A^d}$, and so it is equal to $0$
by \cite[Corollary 1.7]{Sc80}. This proves (c).
\end {proof}

\section{The ideal-triangularizability}

We first extend \cite[Theorem 3]{MKP} by showing the implication that remained unproven in \cite{MKP}. 
In the proof we will make use of the following extension of Ringrose's Theorem; see \cite[Theorem 2.1]{Drn3}.
Recall that if $T$ is a power-compact operator (i.e., some power of $T$ is a compact operator) on a Banach space $X$, 
then the {\it algebraic multiplicity} $m(T, \lambda)$ of a non-zero complex number $\lambda$ is the dimension of the subspace 
${\rm ker \,} ((\lambda - T)^k)$, where $k$ is the smallest natural number such that 
${\rm ker \,} ((\lambda - T)^k) = {\rm ker \,} ((\lambda - T)^{k+1})$. 
On the other hand, the {\it geometric multiplicity} of $\lambda$ is the dimension of the subspace  ${\rm ker \,} (\lambda - T)$.
We say that a chain $\cC$ of closed subspaces of $X$ is a {\it complete} chain if it contains arbitrary intersections and closed linear spans of its members. 
If a closed subspace $\cM$ is in a complete chain $\cC$, then the {\it predecessor}
$\cM_{-}$ of $\cM$ in $\cC$ is defined as the closed linear span of all proper subspaces of $\cM$ belonging to $\cC$.  

\begin{theorem}
\label{Ringrose} 
Let $T$ be a power-compact operator on a Banach space $X$, and let $\cC$ be a complete chain of closed subspaces
invariant under $T$.  Let $\cC^\prime$ be a subchain of $\cC$ of all subspaces $\cM \in \cC$ such that $\cM_{-} \neq \cM$.
For each $\cM \in \cC^\prime$, define $T_{\cM}$ to be the quotient operator
on $\cM / \cM_{-}$ induced by $T$. Then 
$$ \sigma(T) \setminus \{0\} =  \bigcup_{\cM \in \cC^\prime} \sigma(T_{\cM}) \setminus \{0\}. $$
Moreover, for each non-zero complex number $\lambda$ we have 
$$ m(T, \lambda) = \sum_{\cM \in \cC^\prime} m(T_{\cM}, \lambda)  . $$
\end{theorem}

The following theorem completes \cite[Theorem 3]{MKP}. It extends \cite[Theorem 3.8]{BMR}, where the implication $(c) \Rightarrow (a)$ 
has been proved for a positive compact operator on a Banach lattice $l^p$ ($1<p<\infty$). It is also a generalization of 
\cite[Corollary 27]{MR05}, where positive trace-class operators acting on $L^p(\mu)$ ($1 \le p < \infty$) are considered.
A {\it diagonal entry} of an operator $T$ on a Banach lattice $L$ is the number $\varphi_a(T a)$, where $\varphi_a$ 
is the linear functional associated to an atom $a \in L^+$.

\begin{theorem} 
\label{complete_decomposability2}
Let $T$ be a positive power-compact operator on a Banach lattice $L$ with order continuous norm. The following conditions are mutually equivalent: 
\begin{enumerate}
\item [(a)] $T$ is ideal-triangularizable;
\item [(b)] $T - \cD(T)$ is quasinilpotent;
\item [(c)] The diagonal entries of $T$ consists precisely of eigenvalues (except maybe zero) of the operator $T$ repeated according to their algebraic multiplicities. 
 \end {enumerate}
\end{theorem}

\begin {proof}
It suffices to prove that (c) implies (a), since other implications were proved in \cite{MKP}.
Let $\mathcal C$ be a maximal chain of closed ideals invariant under the operator $T$. We will prove that, in fact, $\mathcal C$
is an ideal-triangularizing chain for $T$. Maximality of $\mathcal C$ implies that $\{0\}$ and $L$ are elements of $\mathcal C$, 
and that $\mathcal C$ is a complete chain.
 Therefore, we only need to prove that for every $J \in \mathcal C$  the dimension of the quotient space $J/J_-$ is less than or equal to one. 
Maximality of $\mathcal C$ also implies that the induced operator $T_{J}$ is ideal-irreducible on $J/J_-$, and so 
it is not quasinilpotent by \cite[Theorem 1.3]{DK} provided $\dim(J/J_-) \geq 2$.  Note that $T_J$ can be identified 
with the compression of $T$ to the band  $J \cap J_-^d$.

Let $\mathcal B$ be the set of all ideals $J$ in $\mathcal C$ with $\dim(J/J_-) \geq 2$. Assume that $\mathcal B$ is not empty.
Let $\lambda$ be the largest number among the diagonal entries of the compressions of the operator $T$ onto $J\cap J_-^d$, 
where $J$ runs over $\mathcal B$. Pick any ideal $J \in \mathcal B$ such that $\lambda$ is the diagonal entry 
of the compression of the operator $T$ onto $J\cap J_-^d$. We claim that $\lambda=r(T_{J}).$ Otherwise we would have $\lambda<r(T_{J})$, 
and so the number of appearances of $r(T_{J})$ as a diagonal entry of the compression of $T$ onto some one-dimensional quotient space 
(induced by an ideal in $\mathcal C$) would be precisely the algebraic multiplicity of $r(T_{J})$.
In view of Theorem \ref{Ringrose} this would imply that $r(T_{J})$ is not in the spectrum of $T_{J}$ which is absurd. 
So, $\lambda=r(T_{J})$ as claimed.

Therefore, the spectral radius $r(T_J)$ appears on the diagonal of the compression of $T$ onto $J \cap J_-^d$.
By \cite[Corollary 1]{MKP}, this compression is ideal-reducible, contradicting the maximality of the chain $\mathcal C$. 
This implies that $\mathcal B$ is an empty set, and so $\mathcal C$ is an ideal-triangularizing chain for the operator $T$. 
\end {proof}

Now we turn our attention to multiplicative semigroups of ideal-triangularizable positive operators. We first recall \cite[Theorem 4.5]{Drn1}. 

\begin{theorem}
\label{quasinilpotent}
Every semigroup of quasinilpotent positive compact operators on a Banach lattice is ideal-triangularizable.
\end{theorem} 

As a consequence we obtain the following result.

\begin{theorem} 
\label{atomless}
Let $L$ be an atomless Banach lattice with order continuous norm. If $\cS$ is a semigroup of ideal-triangularizable 
positive compact operators on $L$, then it is ideal-triangularizable. 
\end{theorem}

\begin{proof}
By \cite[Proposition 4.5]{DKan}, every operator in $\cS$ is quasinilpotent, and so Theorem \ref{quasinilpotent} can be applied.
\end{proof}

Theorems \ref{quasinilpotent} and \ref{atomless} do not hold without the assumption that the operators are compact. 
Namely, in \cite{DK02} there was constructed an irreducible semigroup of square-zero positive operators on 
$L^p[0,1)$ $(1\leq p<\infty)$ in which any finite number of elements generate an ideal-triangularizable semigroup.
So, one may seek for extensions of Theorem \ref{atomless} by relaxing the assumption that $L$ is atomless.
It will turn out that  in this case the atomic diagonal operator plays important role.

Let $L$ be a Banach lattice with order continuous norm, and let $\cS$ be a semigroup of positive operators on $L$.
If $\cS$ is ideal-triangularizable, then $\cD(S T) = \cD(S) \cD(T)$ for every pair $\{S, T\} \subseteq \cS$, as 
was shown in the proof of \cite[Proposition 3]{MKP}. 
It follows that $\cD(S T) = \cD(T S)$ or equivalently $\cD(S T - T S) = 0$ for every pair $\{S, T\} \subseteq \cS$.
The following main theorem of the paper treats the question when the converse implication holds. 
Note also that it is a generalization of Theorem \ref{atomless}.

\begin{theorem} 
\label{semigroup1}
Let $\cS$ be a semigroup of ideal-triangularizable positive compact operators on a Banach lattice $L$ with order continuous norm  
such that $\cD(S T) = \cD(T S)$ for every pair $\{S, T\} \subseteq \cS$. Then the semigroup $\cS$ is ideal-triangularizable. 
\end{theorem}

\begin{proof}
With no loss of generality we may assume that $\cS$ is closed under positive scalar multiplication.
Let us prove that the assumptions are satisfied for the closure of $\cS$, so that we may also assume 
that $\cS$ is a closed set. To this end, let a sequence $\{S_k\}_k \subseteq \cS$ converge to a positive compact operator $S$.
By \cite[Theorem 1.6]{Sc80}, $\cD(S_k)$ is a positive compact operator for each $k$.
Since the diagonal operator $\cD$ is continuous by Theorem \ref{zvez diag}, the sequence
$\{\cD(S_k)\}_k$ converges to the operator $\cD(S)$,  so  that the sequence 
$\{S_k-\cD (S_k)\}_k$ converges to the operator $S-\cD (S)$.
Since the operators $S_k-\cD (S_k)$ are all quasinilpotent by Theorem \ref{complete_decomposability2}, the compact operator 
$S-\cD(S)$ is quasinilpotent as well by \cite[Corollary 7.2.11.]{RaRo}. Thus, Theorem \ref{complete_decomposability2} implies that $S$ is ideal-triangularizable. Since $\cD(S T) = \lim_{k \rightarrow \infty} \cD(S_k T) = \lim_{k \rightarrow \infty} \cD(T S_k) = \cD(T S)$ 
for every $T \in \cS$, we conclude (because of symmetry) that $\cD(S T) = \cD(T S)$ for every pair $\{S, T\}$ in the closure of $\cS$.
Therefore, we have shown that there is no loss of generality in assuming that $\cS$ is a closed set.

We first show that $\mathcal S$ is ideal-reducible. Assume otherwise. It follows from Theorem \ref{quasinilpotent}
that any semigroup consisting of quasinilpotent positive compact operators
and multiples of the identity operator is ideal-triangularizable. 
Therefore, there exists an operator $A \in \cS$ such that $r(A)=1$ and $A$ is not the identity operator $I$. 
Since $A$ is ideal-triangularizable, \cite[Proposition 4]{MKP} implies that its spectrum $\sigma(A)$ is contained in $[0,1]$.
We distinguish two cases: 

{\bf Case 1}: The geometric multiplicity of the eigenvalue $1 \in \sigma(A)$ is equal to its algebraic multiplicity.
Then the sequence $\{A^k\}_k$ converges to a nonnegative idempotent $E \in \cS$ of finite rank satisfying $E \neq I$. 
Since $E$ is ideal-triangularizable, it follows from Proposition \ref{o idempotentu 3} that there exist bands $B_1$, $B_2$ and $B_3$ such that,
with respect to the decomposition $L = B_1 \oplus B_2 \oplus B_3$, $E$ has the block-triangular form 
$$ E = \left[ \begin {array}{ccc}
            0 & X & X Y \\
            0 & I & Y \\
            0 & 0 & 0 \\
\end {array}  \right] , 
$$
where only one of the bands $B_1$ and $B_3$ may be equal to zero, and the dimension $n = \dim B_2$ is finite, so that 
$B_2$ is atomic and isomorphic to $\mathbb R^n$ by \cite[Corollary 1, p.70]{Sch}. 

Assume first that $B_1 = \{0\}$ and $Y = 0$, so that $E$ has the form
$$ E =  \left[ \begin {array}{cc}
            I & 0 \\
            0 & 0 \\
\end {array} \right].  $$
Let 
$$ 
S = \left[ \begin {array}{cc}
            S_1 & S_2 \\
            S_3 & S_4 \\
\end {array} \right] \in \cS  
\ \ \ {\rm and} \ \ \ \ 
T = \left[ \begin {array}{cc}
            T_1 & T_2 \\
            T_3 & T_4 \\
\end {array} \right]   \in \cS
$$
be arbitrary operators in $\cS$. The operators $ES \cdot TE$ and $TE\cdot ES$ are both in $\mathcal S$, and they have the same diagonal. From 
$$  
E S \cdot T E = \left[ \begin {array}{cc}
            S_1 T_1 + S_2 T_3 & 0 \\
                 0  &  0  \\
\end {array} \right]
\ \ \ {\rm and} \ \ \ \ 
T E \cdot E S = \left[ \begin {array}{cc}
            T_1 S_1 & T_1 S_2 \\
            T_3 S_1 & T_3 S_2 \\
\end {array} \right]
$$
we conclude that 
$$ \cD( S_1 T_1) + \cD(S_2 T_3)= \cD(T_1S_1). $$
Since 
$$ \mathrm{tr}(\cD(S_1T_1))=\mathrm{tr}(S_1T_1)=\mathrm{tr}(T_1S_1)=\mathrm{tr}(\cD(T_1S_1)) , $$
we obtain that $\mathrm{tr}(\cD(S_2 T_3)) = 0$, so that $\cD(S_2 T_3) = 0$. 
Since $\mathcal S$ is ideal-irreducible, we can choose an atom $a \in B_2$ and an operator $T \in \mathcal S$ such that 
$f = T_3 a$ is a nonzero positive vector in $B_3$. If $\varphi_a$ is the linear functional associated to the atom $a$,
then $\varphi_a(Sf) = \varphi_a(S_2T_3a)+ \varphi_a(S_4T_3a) = 0$, as $\cD(S_2T_3)=0$ and $S_4T_3a \perp a$.
This is a contradiction with Proposition \ref{o razcepnosti}.

The case when $B_3 = \{0\}$ and $X=0$ can be handled similarly. Hence, it remains to consider 
the case when $X$ and $Y$ are not both zero. We consider only the case $X \neq 0$, since the case $Y \neq 0$ is similar.
Let
$$S = \left[ \begin {array}{ccc}
            S_1 & S_2 & S_3 \\
            S_4 & S_5 & S_6 \\
            S_7 & S_8 & S_9 \\
\end {array} \right]  \in \cS
$$ be an arbitrary operator. 
Then 
$$E S= \left[ \begin {array}{ccc}
            XS_4+XYS_7 &  * &  * \\
            * & S_5+YS_8 & * \\
            0 &  0 &  0 \\
\end {array} \right] $$
and
$$S E= \left[ \begin {array}{ccc}
            0 & * & *\\
            0 & S_4X+S_5 & * \\
            0 & * & S_7XY+S_8Y \\
\end {array} \right].$$
Since the operator $E S$ is ideal-triangularizable, its $(1,1)$ block $XS_4+XYS_7$ is ideal-triangularizable as well, and so is the operator $XS_4$.
The equality $\cD(SE)=\cD(ES)$ implies that $\cD(XS_4)=0$, and so the operator $XS_4$ is quasinilpotent by Theorem \ref{complete_decomposability2}.
Since $\sigma(XS_4)\backslash\{0\}=\sigma(S_4X)\backslash\{0\}$, $S_4X$ is a nilpotent operator on a finite-dimensional Banach lattice $B_2$.
Since the operator $S E$ is ideal-triangularizable, its $(2,2)$ block $S_4 X+S_5$ is ideal-triangularizable as well, and so is $S_4X$.
It follows that $\cD(S_4X)=0$. Since $X \neq 0$, there is an atom $a \in B_2$ such that $f = X a$ is a nonzero vector in $B_1$.
The vectors $S_1f$ and $S_7f$ are disjoint with the atom $a$, and so 
$\varphi_a(S f) =  \varphi_a(S_4 f) = \varphi_a(S_4Xa)= 0$, as $\cD(S_4 X)=0$. This is again a contradiction with Proposition \ref{o razcepnosti}.

{\bf Case 2}: The geometric multiplicity of the eigenvalue $1 \in \sigma(A)$ is smaller than its algebraic multiplicity.
Then, by the Riesz decomposition theorem, the Banach space $L$ can be decomposed into a direct sum of a finite-dimensional subspace $L_1$ and 
a (possibly zero) closed subspace $L_2$ such that $A$ has the block-diagonal form
$$ 
A = \left[ \begin {array}{cc}
            I + N & 0 \\
               0  & C \\
\end {array} \right], 
$$
where $N$ is a nilpotent operator on $L_1$ of the nilpotency index $k \geq 2$, and $r(C) < 1$ provided $L_2$ is nonzero. Now
$$ \lim_{m \rightarrow \infty} \frac{A^m}{{m \choose k-1}} = 
\left[ \begin {array}{cc}
             N^{k-1} & 0 \\
               0    &  0 \\
\end {array} \right] \in \cS . $$
This proves that the semigroup $\cS$ contains a square-zero operator $M \neq 0$. 
Let $\mathcal N(M)$ and $\mathcal R(M)$ denote the absolute kernel and the range ideal of the operator $M$, respectively. 
We claim that $\mathcal R(M)\subseteq \mathcal N(M).$
If $y$ is a vector in $\mathcal R(M)$, then there exist positive vectors $x_1,\ldots, x_n\in E$ and positive scalars 
$\lambda_1,\ldots, \lambda_n$ such that 
$$ |y|\leq \sum_{j=1}^n \lambda_j Mx_j . $$
Then we have 
$$ 0 \leq M|y|\leq \sum_{j=1}^n \lambda_j M^2x_j=0 , $$ 
which proves the claim. Now, if $z \in {\mathcal N}(M)^d$ then $M z \in \mathcal R(M) \subseteq \mathcal N(M)$.
This shows that, with respect to the band decomposition $L=\mathcal N(M)\oplus \mathcal N(M)^d$, we have 
$$ M =\left[\begin {array}{cc}
0 & X\\
0 & 0\\
\end {array}\right] . $$

Let $S$ be an arbitrary operator in $\mathcal S$, and let 
$$ S = 
\left[ \begin {array}{cc}
            S_1 & S_2 \\
            S_3 & S_4 \\
\end {array} \right] 
$$ 
be its block operator matrix with respect to the band decomposition 
$L=\mathcal N(M)\oplus \mathcal N(M)^d$. 
Then 
$$  M S =\left[ \begin {array}{cc}
            X S_3 & X S_4 \\
            0 & 0 \\
\end {array} \right]  \qquad \textrm{and}\qquad 
S M =\left[ \begin {array}{cc}
            0 & S_1 X  \\
            0 & S_3 X\\
\end {array} \right] .$$
Since $\cD(MS) = \cD(S M)$, we have $\cD(X S_3)=0$ and $\cD(S_3 X)=0.$
Since $MS$ and $SM$ are ideal-triangularizable operators, the operators $X S_3$ and $S_3 X$ are also ideal-triangularizable.
Theorem \ref{complete_decomposability2} implies that $X S_3$ and $S_3 X$ are both quasinilpotent, so that $MS$ and $SM$ are quasinilpotent as well. 

Let $\mathcal J$ be the semigroup ideal in $\mathcal S$ generated by the operator $M$. 
Pick any $S\in \mathcal J$. Then $S$ is a product of operators from $\mathcal S$, and the operator $M$ appears in this product 
at least once.  If $M$ appears either at the beginning or at the end of this product, then $S$ is quasinilpotent by the observation above. Otherwise there exist $S_1$ and $S_2$ in $\mathcal S$ such that  $S=S_1MS_2.$ 
By the well known equality for the spectral radius, we have 
$r(S)=r(S_1MS_2)=r(MS_2S_1)=0$. This implies that $\mathcal J$ consists of quasinilpotent compact operators, so that it is ideal-triangularizable by Theorem \ref{quasinilpotent}. This is again a contradiction with Proposition \ref{o razcepnosti}.

We have shown that the semigroup $\mathcal S$ is ideal-reducible. 
We will finish the proof by applying the Ideal-triangularization lemma. 
The property $\mathcal D(ST)=\mathcal D(TS)$ is inherited by ideal-quotients. 
Indeed, let $J$ be a closed ideal in $L$ that is invariant under the semigroup $\mathcal S$. Let $S|_{J}$ denote the restriction of the operator $S$ to $J$, and let $P_{J}$ be the band projection onto $J$. 
Then we have
$$\mathcal D(S|_{J}T|_{J})=\sum_{a\in \mathcal A\cap J}P_a(S|_{J})(T|_{J})P_a=\sum_{a\in \mathcal A\cap J}P_a(ST)|_{J}P_a=$$
$$=\sum_{a\in \mathcal A\cap J}P_aP_{J}STP_{J}P_a=P_{J}\left(\sum_{a\in \mathcal A}P_aSTP_a\right)P_{J}=P_{J}\cD(ST)P_{J}.$$
This implies that
$$\mathcal D(S|_{J}T|_{J})=P_{J}\cD(ST)P_{J}=P_{J}\cD(TS)P_{J}=\mathcal D(T|_{J}S|_{J}).$$
This proves that the property of zero diagonals of commutators of operators from $\mathcal S$  is inherited to restrictions on invariant closed ideals.  
Since the induced operator $T_{J}$ of the operator $T$ on the quotient Banach lattice $E/J$ can be identified with the compression of the operator $T$ to $J^d$, we can (similarly as above) prove that the property that the diagonal of every commutator of operators from $\mathcal S$ is zero is inherited by induced operators on the quotient Banach lattice $E/J$. 
Since ideal-triangularizability is inherited by ideal-quotients by \cite[Proposition 2.3]{kanna}, we finish the proof by applying the Ideal-triangularization lemma.
\end{proof}

In the finite-dimensional case Theorem  \ref{semigroup1} can be stated as follows.

\begin {corollary}
\label{matrices}
Let $\mathcal S$ be a semigroup of nonnegative $n\times n$ matrices. Suppose that for every matrix $S\in \mathcal S$ there exists a permutation matrix $P_S$ (depending on $S$) such that the matrix $P_S S P_S^{-1}$ is upper triangular. 
If $\cD(S T) = \cD(T S)$ for every pair $S,T \in\mathcal S$, then there exists a permutation matrix $P$ such that every matrix 
in the semigroup $P\mathcal S P^{-1}$ is upper triangular. 
\end {corollary}

The following example which was already introduced in \cite{DJK} shows that Theorem \ref{semigroup1} and Corollary \ref{matrices} do not hold for collections of ideal-triangularizable nonnegative matrices with zero diagonals of commutators.

\begin {example}
{\rm
Let $e_1$, $e_2$, $\ldots$, $e_n$ be the standard basis vectors of $\RR^n$, where $n \ge 3$.
Define ideal-triangularizable nilpotent matrices by $A_i = e_i e_{i+1}^T$ 
for $i=1, 2, \ldots, n-1$, and $A_n =  e_n e_1^T$. Then the collection $\{A_1, A_2, \ldots, A_n\}$ has the property that 
$\cD(A_i A_j)= 0$ for all $1\leq i,j\leq n$. We claim that the collection is not ideal-triangularizable. Assume the contrary. 
Then the sum $S= A_1+A_2+\ldots+A_n$ is ideal-triangularizable. Since all the diagonal entries of $S$ are zero, $S$ must be nilpotent 
which contradicts the fact that $S^n = I$.}
\end {example}

The following example shows that Theorem \ref{semigroup1} and Corollary \ref{matrices} do not hold without the assumption that operators are positive.

\begin{example} 
{\rm 
Let
$$ 
A = \left[ \begin {array}{cccc}
            0 & 0 & 1 & -1 \\
            0 & 0 & -1 & 1 \\
            0 & 0 & 0 & 0 \\
            0 & 0 & 0 & 0 \\
           \end {array} \right] 
\ \ \ {\rm and} \ \ \ \ 
B = \left[ \begin {array}{cccc}
            0 & 0 & 0 & 0 \\
            0 & 0 & 0 & 0 \\
            1 & 1 & 0 & 0 \\
            1 & 1 & 0 & 0 \\
           \end {array} \right] .           
$$
Then $A^2 = B^2 = A B = B A = 0$, so that $\cS = \{0, A, B\}$ 
is a semigroup of ideal-triangularizable matrices such that $\cD(S) = 0$ for all $S \in \cS$. 
However, $\cS$ is not ideal-triangularizable, as the diagonal of the matrix 
$$ 
|A| + B = \left[ \begin {array}{cccc}
            0 & 0 & 1 & 1 \\
            0 & 0 & 1 & 1 \\
            1 & 1 & 0 & 0 \\
            1 & 1 & 0 & 0 \\
           \end {array} \right] 
$$          
is zero,  but the matrix is not nilpotent. }
\end{example}

{\it Acknowledgments.} This work was supported in part by the Slovenian Research Agency. \\

\bigskip
		
\noindent
     Roman Drnov\v{s}ek, Marko Kandi\'{c} : \\
     Faculty of Mathematics and Physics \\
     University of Ljubljana \\
     Jadranska 19 \\
     1000 Ljubljana \\
     Slovenia \\[1mm]
     E-mails : roman.drnovsek@fmf.uni-lj.si, marko.kandic@fmf.uni-lj.si 


\begin{thebibliography}{9999}

\bibitem{AbAl}  Y. A. Abramovich, C. D. Aliprantis, 
\textit{An Invitation to Operator Theory}. American Mathematical Society, Providence (2002).

\bibitem{AlAp} C. D. Aliprantis, O. Burkinshaw, 
\textit{Positive operators}, Reprint of the 1985 original, Springer, Dordrecht (2006).  

\bibitem{BMR} J. Bernik, L. W. Marcoux, H. Radjavi, 
\textit{Spectral conditions and band reducibility of operators}, 
J. Lond. Math. Soc. (2) {\bf 86} (2012), no. 1, 214--234. 

\bibitem{Drn3} R. Drnov\v{s}ek, {\it Once more on positive commutators}, 
Studia Math. {\bf 211} (2012), no. 3, 241--245. 

\bibitem{Drn1} R. Drnov\v{s}ek, {\it Common invariant subspaces for collections of operators},
Integral Equat. Oper. Th. {\bf 39}  (2001), 253--266. 

\bibitem{Drn} R. Drnov\v{s}ek, {\it Triangularizing semigroups of positive operators on an atomic normed Riesz spaces},
 Proc. Edin. Math. Soc. {\bf 43}  (2000), 43--55.   

\bibitem{DJK} R. Drnov\v sek, M. Jesenko, M. Kandi\' c, {\it Positive commutators and collections of operators},
Oper. Matrices \textbf{6} (3) (2012),  535--542.

\bibitem{DKan} R. Drnov\v sek, M. Kandi\'{c}, {\it Ideal-triangularizability of semigroups of positive operators},
Integral Equat. Oper. Th.  \textbf{64} (4) (2009), 539--552.

\bibitem{DK} R. Drnov\v sek, M. Kandi\'{c}, {\it More on positive commutators},
J. Math. Anal. Appl. \textbf{373} (2011), 580 -- 584.

\bibitem{DK02} R. Drnov\v sek, D. Kokol-Bukov\v sek, L. Livshits, 
G. MacDonald, M. Omladi\v c, H. Radjavi,  
{\it An irreducible semigroup of non-negative square-zero operators},
Integral Equat. Oper. Th. {\bf 42} (2002), no. 4, 449--460.

\bibitem{kanna} M. Kandi\' c,
{\it Ideal-triangularizability of upward directed sets of positive operators},
Ann. Funct. Anal. {\bf 2} (1) (2011), 206--219.

\bibitem {MKP} M. Kandi\' c, \textit{Multiplicative coordinate functionals
and ideal-triangularizability}, Positivity, doi:10.1007/s11117-013-0222-z. 

\bibitem{LuxZaa} W. A. J. Luxemburg, A. C. Zaanen, 
{\it Riesz spaces I}, North-Holland, Amsterdam (1971).

\bibitem{MR05} G. MacDonald, H. Radjavi,
{\it Standard triangularization of semigroups of non-negative operators},
J. Funct. Anal. {\bf 219} (2005), 161--176.

\bibitem{RaRo} H. Radjavi, P. Rosenthal,  \textit{Simultaneous Triangularization}, 
Springer-Verlag, New York, 2000.

\bibitem{Sch} H. H. Schaefer,
{\it Banach lattices}, Springer-Verlag, Berlin-Heidelberg-New York, 1974.

\bibitem{Sc80} A. R. Schep, 
{\it Positive diagonal and triangular operators}, J. Oper. Theory {\bf 3} (1980), 165--178.

\bibitem{Vo88} J. Voigt, \textit{The projection onto the center of operators in a Banach lattice},
Math. Z. \textbf{199} (1988), no. 1, 115--117.

\bibitem {Zhong} Y. Zhong, 
{\it Functional positivity and invariant subspaces of semigroups of operators}, 
Houston J. Math. {\bf 19} (1993), 239--262.
  
\end{thebibliography}
\end{document}